\documentclass[conference]{IEEEtran}
\IEEEoverridecommandlockouts
\usepackage{cite}
\usepackage{amsmath,amssymb,amsfonts}
\usepackage{algorithmic}
\usepackage{graphicx}
\usepackage{textcomp}
\usepackage[dvipsnames]{xcolor}
\def\BibTeX{{\rm B\kern-.05ema{\sc i\kern-.025em b}\kern-.08em
    T\kern-.1667em\lower.7ex\hbox{E}\kern-.125emX}}

\usepackage[utf8]{inputenc}
\usepackage[T1]{fontenc}
\usepackage[english]{babel} 
\usepackage{amsmath}
\usepackage{amsfonts}
\usepackage{amssymb}
\usepackage{amsthm}
\usepackage{makeidx}
\usepackage{graphicx}
\usepackage{lmodern}

\usepackage{tikz}
\usetikzlibrary{scopes, fadings}
\usetikzlibrary{shapes.misc}
\usetikzlibrary{decorations.pathreplacing}
\usetikzlibrary{patterns}
\usetikzlibrary{arrows,shapes,positioning}
\usetikzlibrary{shapes.multipart}
\usetikzlibrary {arrows.meta} 

\usepackage{comment}
\usepackage{calrsfs}
\DeclareMathAlphabet{\pazocal}{OMS}{zplm}{m}{n}
\usepackage{bm}

\usepackage{fourier-orns}
\usepackage{mathtools}
\usepackage{relsize}
\usepackage{dsfont}
\usepackage{comment}

\newcommand{\R}{\mathbb{R}}

\newcommand{\Gpazo}{\pazocal{G}}

\newcommand{\Vpazo}{\pazocal{V}}

\newcommand{\Lpazo}{\pazocal{L}}

\newcommand{\Dpazo}{\pazocal{D}}

\newcommand{\Spazo}{\pazocal{S}}

\newcommand{\Lip}{\textnormal{Lip}}

\newcommand{\loc}{\textnormal{loc}}

\newcommand{\textbn}[1]{\textnormal{\textbf{#1}}}
\newcommand{\co}{\overline{\textnormal{co}} \hspace{0.05cm}}

\newcommand{\xb}{\boldsymbol{x}}
\newcommand{\yb}{\boldsymbol{y}}

\newcommand{\Ab}{\boldsymbol{A}}

\newcommand{\Lb}{\boldsymbol{L}}
\newcommand{\Kb}{\boldsymbol{K}}
\newcommand{\Qb}{\boldsymbol{Q}}

\newcommand{\Db}{\boldsymbol{D}}

\renewcommand{\epsilon}{\varepsilon}

\newcommand{\INTSeg}[4]{\int_{#3}^{#4} #1 \textnormal{d} #2}

\newcommand{\derv}[2]{\frac{\textnormal{d} #1}{ \textnormal{d} #2}}

\newtheorem{lem}{Lemma}
\newtheorem{Def}{Definition}
\newtheorem{thm}{Theorem}
\newtheorem{prop}{Proposition}

\newtheorem{taggedhypx}{Hypotheses}
\newenvironment{taggedhyp}[1]
    {\renewcommand\thetaggedhypx{#1}\taggedhypx}
    {\endtaggedhypx}

\begin{document}

\title{Exponential Consensus Formation in Time-Varying Multiagent Systems via Compactification Methods}

\author{\IEEEauthorblockN{1\textsuperscript{st} Beno\^it Bonnet-Weill}
\IEEEauthorblockA{\textit{Laboratoire des Signaux et Syst\`emes,} \\ \textit{Universit\'e Paris-Saclay, CNRS,} \\
\textit{CentraleSupélec, 91190 Gif-sur-Yvette, France.}\\
\textit{E-mail}: {\tt benoit.bonnet-weill@centralesupelec.fr}}
\and
\IEEEauthorblockN{2\textsuperscript{nd} Mario Sigalotti}
\IEEEauthorblockA{\textit{Inria Paris and Laboratoire Jacques-Louis Lions,} \\
\textit{Sorbonne Université, Université Paris-Diderot SPC,} \\ 
\textit{CNRS, Inria, 75005 Paris, France}  \\
\textit{E-mail}: {\tt mario.sigalotti@inria.fr}}
}

\maketitle

\begin{abstract}
In this article, we establish exponential contraction results for the diameter and variance of general first-order multiagent systems. Our approach is based on compactification techniques, and works under rather mild assumptions. Namely, we posit that either the scrambling coefficient, or the algebraic connectivity of the averaged interaction graphs of the system over all time windows of a given length are uniformly positive. 
\end{abstract}



\section{Introduction}

The study of self-organisation in cooperative dynamics has been a prominent topic in multiagent system analysis and network theory for several decades. Since the seminal works of DeGroot \cite{DeGroot1974}, Vicksek \cite{Vicsek1995} and later of Hegselmann and Krause \cite{Hegselmann2002}, an extensive literature has been preoccupied with finding sharp sufficient conditions for the emergence of global clustering -- called \textit{consensus} -- in first-order multiagent dynamics. Some of the farthest reaching of such contributions can be attributed to Moreau \cite{Moreau2005}, Olfati-Saber and Murray \cite{Olfati2004} and Jadbabaie et.al \cite{Jadbabaie2003}, who established very general conditions bearing on the connectivity properties of the interactions between agents entailing such asymptotic behaviour. 

Following the introduction of the now-called Cucker--Smale alignment model in \cite{CS2}, the investigation of clustering patterns in cooperative dynamics spread further to communities working at the interface between particle systems and partial differential equations \cite{Carrillo2010,HaLiu}, with the aim of studying macroscopic approximations of cooperative dynamics and to derive scale-free clustering conditions by means of energy methods. In this context, the two main candidate Lyapunov functions have historically been the \textit{variance} on the one hand, which leads to an analogue of the $L^2$-stability theory for multiagent dynamics, and the \textit{diameter} on the other, which provides $L^{\infty}$-type estimates. Besides, it has long been known that the decay at infinity of both functions -- and in particular its exponential character or lack thereof -- is tightly linked to relevant graph-theoretic quantities related to the connectivity properties of the underlying interaction graphs, namely the \textit{algebraic connectivity} introduced by Fiedler \cite{Fiedler1973} and the scrambling coefficient \cite{Seneta1979} arising in the stability theory of stochastic matrices. We point the interested readers to the excellent survey \cite{Motsch2014} for a complete description of the foundations of the finite-dimensional theory, and to the work \cite{ConsensusGraphon} by the authors where we generalise the underlying core results to infinite-dimensional graphon dynamics. We also quote \cite{Boudin2022}, which keenly builds on the approach of \cite{Olfati2004} to establish the exponential decay of the $L^2$-norm in very general time-independent asymmetric graphon dynamics.  

\medskip

In this article, our goal is to derive exponential contractivity results for multiagent systems of the form
\begin{equation}
\label{eq:IntroMultiagent}
\dot x_i(t) = \frac{1}{N} \sum_{j=1}^N a_{ij}(t) \phi(|x_i(t) - x_j(t)|)(x_j(t) - x_i(t)),
\end{equation}
whose interaction topologies are time-dependent and parsimonious. More precisely, we assume that the underlying graphs are allowed to be very sparse at each time, and only well-connected on average. A relevant way to mathematically translate this idea is to require that the right-hand side of \eqref{eq:IntroMultiagent} be \textit{persistent} in a suitable sense. Taking inspiration from the aforedescribed corpus on consensus formation, our persistence conditions take the form of uniformly positive lower bounds on the scrambling coefficient and algebraic connectivity of the average interaction graphs over all time windows of a given length, see Definition~\ref{def:Persistent} for precise mathematical statements. Under this kind of assumption, the exponential decay of the variance was previously established in \cite{CSComFail} for symmetric topologies, by leveraging strict Lyapunov design techniques borrowed from \cite{MazencMalisoff}. However, transposing such an approach to directed interaction topologies or to the derivation of exponential contraction estimates for the diameter proved infeasible so far, as it heavily relied on properties satisfied by symmetric matrices along with the fact that the variance is essentially a Hilbert seminorm.

In what follows, we successfully circumvent these technical difficulties by elaborating a strategy based on compactification methods. The latter takes root in a general principle due to Fenichel \cite{fenichel}, which guarantees that if a system driven by a class of time-shift invariant and weakly-$^{\star}$ compact signals happens to be attractive, then it must be exponentially stable. In Theorem \ref{thm:Main1} below, we are thus able to show that the diameter of solutions of \eqref{eq:IntroMultiagent} decays exponentially under an average scrambling persistence condition, thus providing a general $L^{\infty}$-counterpart to \cite{CSComFail} while answering an open problem stated in \cite{ConsensusGraphon}. The proof relies on a combination of the compactness result of Lemma~\ref{lem:CompactnessSolutionSet}, together with the regular diameter decay established in Lemma~\ref{lem:DiameterDecay}. This last result essentially contains the main technical difficulty of the approach, namely that having an average interaction topology whose scrambling coefficient is positively lower-bounded entails the strict decay of the diameter over the underlying time interval. In Theorem~\ref{thm:Main2}, we likewise prove that the variance of a linear variant of \eqref{eq:IntroMultiagent} decays exponentially under the assumption that the interaction graphs are \textit{balanced}, and that the algebraic connectivities of their averages over time windows of a preset length are uniformly positively lower-bounded. 

The main interest of our approach and those developed in \cite{CSComFail,ConsensusGraphon,Boudin2022} is that, unlike most preexisting works on consensus formation for heterogeneous multiagent systems bearing on graph-theoretic quantities, our results are based on quantitative energy estimates and, as such, hopefully transposable to infinite-dimensional systems such as meanfield or graphon dynamics, in the spirit of \cite{ConsensusGraphon}. It should again be stressed that the key fact allowing for the strategy developed in this work to function is that signals complying with scrambling or connectivity persistence conditions are stable under weak-$^{\star}$ convergence, as shown in Proposition~\ref{prop:CompactnessPersistent}, a property which may very well fail to hold for many larger classes of signals satisfying less uniform connectivity conditions such as those of \cite{Moreau2005} or \cite{Barabanov2018}. We finally mention that a similar approach has been recently adopted in \cite{AnconaBentaibiRossi}, where consensus is studied under a persistence condition expressing uniform positivity of the average interaction graph.

\medskip

The structure of the paper is the following. In Section~\ref{section:Preliminaries}, we recollect basic facts about multiagent dynamics and graph theory. We then prove the exponential diameter and variance decays in Section~\ref{section:Diameter} and Section~\ref{section:Variance} respectively.


\section{Preliminaries}
\label{section:Preliminaries}


\subsection{Multiagent dynamics and some of their properties}

In this section, we recollect some elementary facts about first-order nonlinear cooperative dynamics of the form 
\begin{equation}
\label{eq:Multiagent}
\left\{
\begin{aligned}
& \dot x_i(t) = \frac{1}{N} \sum_{j=1}^N a_{ij}(t) \phi(|x_i(t)-x_j(t)|)(x_j(t)-x_i(t)), \\
& x_i(0) = x_i^0,
\end{aligned}
\right.
\end{equation}
wherein $x_i(\cdot) \in \Lip_{\loc}(\R_+,\R^d)$ represents the evolution through time of the position of the agent with label $i \in\{1,\dots,N\}$ in $\R^d$. Here $\R_+$ stands for the half-line $[0,+\infty)$, while later in the text we use $\R_+^*$ for $(0,+\infty)$. Throughout this article, we will work under the following modelling assumptions, which are quite standard. 

\begin{taggedhyp}{$\textbn{(H)}$} \hfill
\label{hyp:H}
\begin{enumerate}
\item[$(i)$] The signals $t \in \R_+ \mapsto a_{ij}(t) \in [0,1]$ are Lebesgue measurable for each $i,j \in \{1,\dots,N\}$. 
\item[$(ii)$] The map $\phi : \R_+ \to \R_+^*$ is locally Lipschitz. 
\end{enumerate}
\end{taggedhyp}

In the sequel, for the sake of brevity, we will use the condensed notation $\xb := (x_1,\dots,x_N) \in (\R^d)^N$ to refer to the collection of positions of all the agents, and use $\co(\xb)\subset \R^d$ to denote the polytope defined as the convex hull of the entries of one such element $\xb \in (\R^d)^N$.

\begin{prop}[Elementary stability estimates]
\label{prop:NormEst}
For every solution $\xb(\cdot) \in \Lip_{\loc}(\R+,(\R^d)^N)$ of \eqref{eq:Multiagent}, it holds that
\begin{equation*}
\max_{1 \leq i \leq N} |x_i(s)| \leq \max_{1 \leq i \leq N} |x_i(t)| 
\end{equation*}
as well as 
\begin{equation*}
\co(\xb(s)) \subset \co(\xb(t)) 
\end{equation*}
for all times $s \geq t \geq 0$.
\end{prop}

\begin{proof}
This follows, for instance, from \cite[Propositions~2.10 and 2.11]{ConsensusGraphon}.
\end{proof}

As a direct consequence of these estimates, there exist for each 
compact set $\Kb^0 \subset (\R^d)^N$ a pair of positive constants $0 < c_{\phi} \leq C_{\phi}$ such that 
\begin{equation}
\label{eq:PhiEst}
c_{\phi} \leq \phi(|x_i(t) - x_j(t)|) \leq C_{\phi}
\end{equation}
for each $i,j \in \{1,\dots,N\}$ and every solution $\xb(\cdot) \in \Lip(\R_+,\R^N)$ of \eqref{eq:Multiagent} starting from a datum $\xb^0 \in \Kb^0$. In particular, every solution of \eqref{eq:Multiagent} is globally Lipschitz. 

Given a solution $\xb(\cdot) \in \Lip(\R_+,(\R^d)^N)$ of \eqref{eq:Multiagent}, we denote its \textit{diameter} by
\begin{equation*}
\Dpazo(\xb(t)) := \max_{1 \leq i,j \leq N} |x_i(t) - x_j(t)|,
\end{equation*}
and likewise define its \textit{variance} as
\begin{equation*}
\Vpazo(\xb(t)) := \frac{1}{N} \sum_{i=1}^N |x_i(t) - \bar{\xb}(t)|^2, 
\end{equation*}
where $\bar{\xb}(t) := \sum_{i=1}^N x_i(t)$ is the average position of the system. Both functions are ubiquitous in the literature on collective dynamics pertaining to clustering behaviours (see, e.g., \cite{Motsch2014}) and will feature in our main results.

We end this section by establishing a handy geometric characterisation for the agents maximising the diameter of the system at a given time. 

\begin{prop}[Geometry of points maximising the diameter]
\label{prop:Diameter}
Given a solution $\xb(\cdot) \in \Lip(\R_+,(\R^d)^N)$ of \eqref{eq:Multiagent} and a pair of indices $(i,j) \in \{1,\dots,N\}^2$ such that
\begin{equation*}
\Dpazo(\xb(t)) = |x_i(t) - x_j(t)| 
\end{equation*}
at some time $t \geq 0$, it holds for all $s \geq t$ that 
\begin{equation*}
\langle y , x_i(t) - x_j(t) \rangle < \langle x_i(t) , x_i(t) - x_j(t) \rangle
\end{equation*}
for every $y \in \co(\xb(s)) \setminus \{x_i(t)\}$, and similarly
\begin{equation*}
\langle x_j(t) , x_i(t) - x_j(t) \rangle < \langle y , x_i(t) - x_j(t) \rangle
\end{equation*}
for every $y \in \co(\xb(s)) \setminus \{x_j(t)\}$. 
\end{prop}

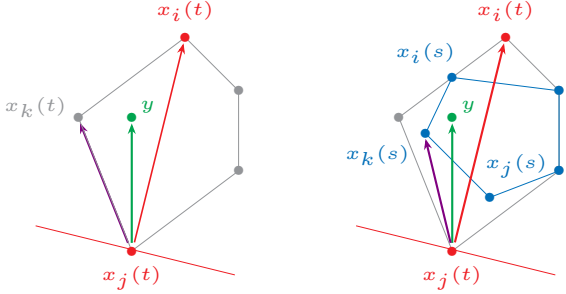
\begin{figure}[!ht]
\resizebox{0.9\linewidth}{!}{
\begin{tikzpicture}
\draw[Red, line width = 0.15mm, -{Stealth[length=1mm]}] (0.025,0.08) -- + (0.46,1.865); 
\draw[Red, line width = 0.05mm, xshift=0.1cm, yshift = -0.015cm] (-1,0.25) --+ (1.865,-0.46);
\draw[Green, line width = 0.2mm, -{Stealth[length=1mm]}] (0,0.075) -- (0,1.195); 
\draw[violet, line width = 0.2mm, -{Stealth[length=1mm]}] (-0.035,0.075) -- (-0.485,1.2);  
\draw[Gray, line width = 0.05mm] plot [smooth cycle, tension=0] coordinates {(0,0) (-0.5,1.25) (0.5,2) (1,1.5) (1,0.75)};
\draw[Red] (0,0) node {\tiny $\bullet$};
\draw[Red] (0.5,2) node {\tiny $\bullet$};
\draw[Gray] (1,1.5) node {\tiny $\bullet$};
\draw[Gray] (1,0.75) node {\tiny $\bullet$};
\draw[Gray] (-0.5,1.25) node {\tiny $\bullet$};
\draw[Green] (0,1.25) node {\tiny $\bullet$}; 
\draw[Red] (0,-0.25) node {\tiny $x_j(t)$};
\draw[Red] (0.5,2.25) node {\tiny $x_i(t)$};
\draw[Gray] (-0.9,1.35) node {\tiny $x_k(t)$};
\draw[Green] (0.15,1.35) node {\tiny $y$}; 
\begin{scope}[xshift = 3cm]
\draw[Gray, line width = 0.05mm] plot [smooth cycle, tension=0] coordinates {(0,0) (-0.5,1.25) (0.5,2) (1,1.5) (1,0.75)};
\draw[RoyalBlue, line width = 0.05mm] plot [smooth cycle, tension=0] coordinates {(0.35,0.5) (-0.25,1.1) (0,1.625) (1,1.5) (1,0.75)};
\draw[Red, line width = 0.2mm, -{Stealth[length=1mm]}] (0.025,0.08) -- + (0.46,1.865); 
\draw[Red, line width = 0.05mm, xshift=0.1cm, yshift = -0.015cm] (-1,0.25) -- + (1.865,-0.46);
\draw[Green, line width = 0.2mm, -{Stealth[length=1mm]}] (0,0.075) -- (0,1.195);
\draw[violet, line width = 0.2mm, -{Stealth[length=1mm]}] (-0.015,0.075) -- (-0.25,1.05);
\draw[Red] (0,0) node {\tiny $\bullet$};
\draw[Red] (0.5,2) node {\tiny $\bullet$};
\draw[Gray] (1,1.5) node {\tiny $\bullet$};
\draw[Gray] (1,0.75) node {\tiny $\bullet$};
\draw[Gray] (-0.5,1.25) node {\tiny $\bullet$};
\draw[Red] (0,-0.25) node {\tiny $x_j(t)$};
\draw[Red] (0.5,2.25) node {\tiny $x_i(t)$};
\draw[NavyBlue] (0.6,0.8) node {\tiny $x_j(s)$};
\draw[NavyBlue] (-0.25,1.85) node {\tiny $x_i(s)$};
\draw[NavyBlue] (-0.7,0.9) node {\tiny $x_k(s)$};
\draw[Green] (0.15,1.35) node {\tiny $y$}; 
\draw[NavyBlue] (0,0) + (0.35,0.5) node {\tiny $\bullet$};
\draw[NavyBlue] (0.5,2) + (-1/2,-0.375) node {\tiny $\bullet$};
\draw[NavyBlue] (1,1.5) node {\tiny $\bullet$};
\draw[NavyBlue] (1,0.75) node {\tiny $\bullet$};
\draw[NavyBlue] (-0.5,1.25) + (0.25,-0.15) node {\tiny $\bullet$};
\draw[Green] (0,1.25) node {\tiny $\bullet$}; 
\end{scope}
\end{tikzpicture}}
\caption{Illustration of the geometric result of Proposition~\ref{prop:Diameter} for a multiagent system snapped at times $t \geq 0$ (left) and $s \geq t$ (right): all scalar products of the form $\langle x_i(t) - x_j(t) , y - x_j(t) \rangle$ remain strictly positive for every point $y \in \co(\xb(s)) \setminus \{x_j(t)\}$ with $s \geq t$.}
\end{figure}

\begin{proof}
We only prove the first inequality, as the second one can be derived by simply exchanging the roles of $i$ and $j$. Note first that by \cite[Lemma 3.4]{ConsensusGraphon}, one has 
\begin{equation*}
\langle x_k(t) , x_i(t) - x_j(t) \leq \langle x_i(t) , x_i(t) - x_j(t) \rangle
\end{equation*}
for all times $t \geq 0$ and every $k \in \{1,\dots,N\}$. Besides, it can be shown by contradiction that the inequality in the previous expression is strict if $x_k(t) \neq x_i(t)$, which yields
\begin{equation*}
\langle y(t) , x_i(t) - x_j(t) \rangle  < \langle x_i(t) , x_i(t) - x_j(t) \rangle
\end{equation*}
for every $y(t) \in \co(\xb(t)) \setminus \{x_i(t)\}$ by linearity. We may then conclude by noting that $\co(\xb(s)) \subset \co(\xb(t))$ for all times $s \geq t$ as a consequence of Proposition~\ref{prop:NormEst}. 
\end{proof}


\subsection{Graph theory and time-dependent interaction topologies}

In this second preliminary section, we recollect basic facts about graphs and their connectivity properties. 

Given some integer $N \geq 1$, we shall consider interaction digraphs represented by their \textit{adjacency matrices} $\Ab_N := (a_{ij})_{1 \leq i,j \leq N} \in \Lpazo(\R^N)$, whose coefficients are real numbers valued in $[0,1]$ satisfying $a_{ii} = 1$ for each $i \in \{1,\dots,N\}$. We recall below the definition of the so-called \textit{scrambling coefficient} for such graphs, see e.g. \cite{Seneta1979}, which are known to be tightly linked to the contractivity properties of stochastic matrices. 

\begin{Def}[Scrambling coefficient]
\label{def:Scrambling}
The \textnormal{scrambling coefficient} of a graph with adjacency matrix $\Ab_N \in \Lpazo(\R^N)$ is defined by 
\begin{equation}
\label{eq:Scrambling}
\eta(\Ab_N) := \min_{1 \leq i,j \leq N} \frac{1}{N} \sum_{k=1}^N \min\{a_{ik} , a_{jk} \}. 
\end{equation}
\end{Def}

In what follows, we say that an adjacency matrix $\Ab_N \in \Lpazo(\R^N)$ is \textit{balanced} if its in- and out-degrees are equal, i.e.
\begin{equation*}
d_i := \sum_{j=1}^N a_{ij} = \sum_{j=1}^N a_{ji} 
\end{equation*}
for every $i \in \{1,\dots,N\}$. We can then associate with such a graph its (normalised) \textit{Laplacian matrix}, defined by 
\begin{equation*}
\Lb_N := \tfrac{1}{N} \big(\Db_N - \Ab_N \big),
\end{equation*}
where $\Db_N := \mathrm{diag}(d_1,\dots,d_N) \in \Lpazo(\R^N)$. This allows us to define the so-called \textit{algebraic connectivity}, which is another quantity of interest when studying the clustering properties of an interaction topology. It was introduced by Fiedler for symmetric graphs in \cite{Fiedler1973}, and later extended to general digraphs by Wu in \cite{Wu2005bis,Wu2005ter}.

\begin{Def}[Algebraic connectivity]
\label{def:Algebraic}
The \textnormal{algebraic connectivity} of a balanced graph with Laplacian matrix $\Lb_N \in \Lpazo(\R^N)$, denoted by $\lambda_2(\Lb_N)$, is the largest real number such that
\begin{equation}
\label{eq:Lambda2Charac}
\frac{1}{2N^2} \sum_{i,j=1}^N a_{ij} |x_i - x_j|^2 \geq \lambda_2(\Lb_N) \Vpazo(\xb)
\end{equation}
for all $\xb \in (\R^d)^N$.
\end{Def}

The main objective of this article is to derive quantitative convergence results for systems of the form \eqref{eq:Multiagent} driven by time-varying graphs $t \in \R_+ \mapsto \Ab_N(t) \in \Lpazo(\R^N)$. To this end, we consider the following \textit{persistence conditions}, which are reminiscent of those in, e.g., \cite{CSComFail,AnconaBentaibiRossi,ConsensusGraphon}.

\begin{Def}[Scrambling- and connectivity-persistent signals]
\label{def:Persistent}
Given a pair of parameters $(\tau,\mu) \in \R_+^* \times (0,1]$, we denote by $\Gpazo_{\eta}(\tau,\mu) \subset L^{\infty}(\R_+,\Lpazo(\R^N))$ the collection of all time-dependent adjacency matrices satisfying
\begin{equation*}
\eta \bigg( \frac{1}{\tau} \INTSeg{\Ab_N(s)}{s}{t}{t+\tau} \bigg) \geq \mu
\end{equation*}
for all times $t \geq 0$. Similarly, we denote by $\Gpazo_{\lambda_2}(\tau,\mu) \subset L^{\infty}(\R_+,\Lpazo(\R^N))$ the collection of all time-dependent balanced adjacency matrices whose graph-Laplacians satisfy 
\begin{equation*}
\lambda_2 \bigg( \frac{1}{\tau} \INTSeg{\Lb_N(s)}{s}{t}{t+\tau} \bigg) \geq \mu
\end{equation*}
for all times $t \geq 0$. 
\end{Def}

Our main results, which are developed in Sections~\ref{section:Diameter} and \ref{section:Variance} respectively, will crucially rely on the following weak-$^{\star}$ compactness property enjoyed by scrambling- and connectivity-persistent signals. 

\begin{prop}[Compactness of persistent signals] 
\label{prop:CompactnessPersistent}
For every $(\tau,\mu) \in \R_+^* \times (0,1]$, the sets $\Gpazo_{\eta}(\tau,\mu)$ and $\Gpazo_{\lambda_2}(\tau,\mu)$ are compact for the weak-$^{\star}$ topology of $L^{\infty}(\R_+,\Lpazo(\R^N))$. 
\end{prop}

\begin{proof}
Noting at first that since every sequence $(\Ab_N^n(\cdot)) \subset \Gpazo_{\eta}(\tau,\mu)$ is uniformly bounded, it follows from the Banach--Alaoglu theorem (see e.g. \cite[Theorem 3.16]{Brezis}) that 
\begin{equation*}
\Ab_N^n(\cdot) ~\underset{n \to +\infty}{\rightharpoonup^*}~ \Ab_N(\cdot)
\end{equation*}
for some $\Ab_N(\cdot) \in L^{\infty}(\R_+,\Lpazo(\R^N))$, along a suitable subsequence. Moreover, it follows from the definition of weak-$^{\star}$ convergence that 
\begin{equation*}
\frac{1}{\tau} \INTSeg{\Ab_N^n(s)}{s}{t}{t+\tau} ~\underset{n \to +\infty}{\longrightarrow}~ \frac{1}{\tau} \INTSeg{\Ab_N(s)}{s}{t}{t+\tau}
\end{equation*}
in $\Lpazo(\R^N)$ for each $t \geq 0$. In particular, the  convergence also holds coefficient-wise, which together with the definition \eqref{eq:Scrambling} of the scrambling coefficient then yields
\begin{equation*}
\begin{aligned}
\eta \bigg( \frac{1}{\tau} \INTSeg{\Ab_N(s)}{s}{t}{t+\tau} \bigg) & = \lim_{n \to +\infty} \eta \bigg( \frac{1}{\tau} \INTSeg{\Ab_N^n(s)}{s}{t}{t+\tau} \bigg) \\
& \geq \mu
\end{aligned}
\end{equation*}
and allows us to conclude that $\Ab_N(\cdot) \in \Gpazo_{\eta}(\tau,\mu)$. On the other hand, if $(\Ab_N^n(\cdot)) \subset \Gpazo_{\lambda_2}(\tau,\mu)$, it then follows from the very same arguments that 
\begin{equation*}
\frac{1}{\tau} \INTSeg{\Lb_N^n(s)}{s}{t}{t+\tau} ~\underset{n \to +\infty}{\longrightarrow}~ \frac{1}{\tau} \INTSeg{\Lb_N(s)}{s}{t}{t+\tau}
\end{equation*}
along a subsequence for all times $t \geq 0$ and some $\Lb_N(\cdot) \in L^{\infty}(\R_+,\Lpazo(\R^N))$. One may then conclude by observing that the algebraic connectivity of a balanced graph is the smallest non-zero eigenvalue of the symmetric part of its Laplacian matrix, see e.g. \cite[Lemma 12]{Wu2005ter}, and then use the fact that the spectrum of a matrix depends continuously thereon, by virtue e.g. of \cite[Appendix D]{Horn1985}.
\end{proof}


\section{Exponential diameter decay for scrambling-persistent topologies}
\label{section:Diameter}

In this section, we prove the main result of this article, which shows that the diameter of any solution of \eqref{eq:Multiagent} driven by a scrambling-persistent signal decays exponentially. 

\begin{thm}[Exponential diameter decay for scrambling-persistent topologies]
\label{thm:Main1}
Given a pair $(\tau,\mu) \in \R_+^* \times (0,1]$ and a compact set $\Kb^0 \subset (\R^d)^N$, there exist positive constants $\alpha,\gamma > 0$ such that
\begin{equation}
\label{eq:ThmDiamDissip}
\Dpazo(\xb(t)) \leq \alpha \Dpazo(\xb^0) \exp(-\gamma t)
\end{equation}
for all $t \geq 0$ and any solution $\xb(\cdot) \in \Lip(\R_+,(\R^d)^N)$ of \eqref{eq:Multiagent} driven by a signal $\Ab_N(\cdot) \in \Gpazo_{\eta}(\tau,\mu)$. 
\end{thm}

As explained above, the proof of this result will be based on a compactification approach and relies crucially on the following topological result. 

\begin{lem}[Compactness of solution sets]
\label{lem:CompactnessSolutionSet}
For every $(\tau,\mu) \in \R_+^* \times (0,1]$ and each compact set $\Kb^0 \subset (\R^d)^N$, the collection of curves
\begin{equation*}
\begin{aligned}
\Spazo_{(\tau,\mu)}^{\eta}(\Kb^0) := \bigg\{ & \xb(\cdot) \in \Lip(\R_+,(\R^d)^N) ~ \text{solving \eqref{eq:Multiagent}} \\
& \; \text{with $\xb^0 \in \Kb^0$ and $\Ab_N(\cdot) \in \Gpazo_{\eta}(\tau,\mu)$}\bigg\}
\end{aligned}
\end{equation*}
is a compact subset of $C^0(\R_+,(\R^d)^N)$ for the topology of uniform convergence on compact sets. 
\end{lem}

\begin{proof}
To begin with, let us fix a sequence of curves $(\xb_n(\cdot)) \subset \Spazo_{(\tau,\mu)}^{\eta}(\Kb^0)$, and observe that $\xb_n(t) \in \Kb^0$ for all times $t \geq 0$ and each $n \geq 1$ as a consequence of Proposition~\ref{prop:NormEst}. In addition, one has for each $\tau,t \geq 0$ that
\begin{equation*}
\begin{aligned}
& |x_i^n(t) - x_i^n (\tau)| \\
& \hspace{0.05cm} \leq  \frac{1}{N} \sum_{j=1}^N \bigg| \INTSeg{a_{ij}^n(s) \phi(|x_i^n(s) - x_j^n(s)|)(x_j^n(s) - x_i^n(s))}{s}{\tau}{t} \, \bigg| \\
& \leq C_{\phi} \mathrm{diam}(\Kb^0) \, |t-\tau|  
\end{aligned}
\end{equation*}
by \eqref{eq:PhiEst}, and the elements of the sequence $(\xb_n(\cdot)) \subset \Lip(\R_+,\R^N)$ are thus uniformly equi-Lipschitz. As they are valued in a common compact set, it follows from the Ascoli--Arzel\`a theorem (see e.g. \cite[Theorem 11.28]{Rudin1987}) that  
\begin{equation}
\label{eq:UnifConvTraj}
\sup_{t \in I} |\xb(t) - \xb_n(t)| ~\underset{n \to +\infty}{\longrightarrow}~ 0  
\end{equation}
for some $\xb(\cdot) \in \Lip(\R_+,(\R^d)^N)$ and over each compact interval $I \subset \R_+$, along a subsequence that we do not relabel. At this stage, note that owing to Proposition~\ref{prop:CompactnessPersistent}, the sequence of signals $(\Ab_N^n(\cdot)) \subset \Gpazo_{\eta}(\tau,\mu)$ associated with the curves $(\xb_n(\cdot)) \subset \Lip(\R_+,(\R^d)^N)$ admits a weak-$^{\star}$ cluster point $\Ab_N(\cdot) \in \Gpazo_{\eta}(\tau,\mu)$. Up to extracting a further subsequence, this implies in particular that  
\begin{equation}
\label{eq:WeakStarConvIntegral}
\INTSeg{a_{ij}^n(t) f(t)}{t}{0}{+\infty} ~\underset{n \to +\infty}{\longrightarrow}~ \INTSeg{a_{ij}(t) f(t)}{t}{0}{+\infty}
\end{equation}
for every function $f(\cdot) \in L^1(\R_+,\R)$. Therefore, by combining \eqref{eq:UnifConvTraj} and \eqref{eq:WeakStarConvIntegral} together with the fact that $\phi : \R_+ \to \R_+^*$ is uniformly continuous on compact sets by Hypothesis~\ref{hyp:H}-$(ii)$, one readily obtains that
{\small 
\begin{equation*}
\begin{aligned}
& \bigg| \INTSeg{a_{ij}(s) \phi(|x_i(s) - x_j(s)|)(x_j(s)-x_i(s))}{s}{0}{t} \\
& \hspace{0.4cm} - \INTSeg{a_{ij}^n(s) \phi(|x_i^n(s) - x_j^n(s)|)(x_j^n(s)-x_i^n(s))}{s}{0}{t} \, \bigg| \\
& \leq \bigg|\INTSeg{(a_{ij}(s) - a_{ij}^n(s)) \phi(|x_i(s) - x_j(s)|)(x_j(s)-x_i(s))}{s}{0}{t} \, \bigg| \\
& \hspace{0.5cm} + \sup_{s \in [0,t]} \Big| \phi(|x_i(s) - x_j(s)|)(x_j(s)-x_i(s)) \\
& \hspace{1.9cm} - \phi(|x_i^n(s) - x_j^n(s)|)(x_j^n(s)-x_i^n(s)) \Big| \underset{n \to +\infty}{\longrightarrow} 0
\end{aligned}
\end{equation*}
}
%
for all times $t \geq 0$. By passing to the limit as $n \to +\infty$ in the integral form of \eqref{eq:Multiagent} satisfied by $\xb_n(\cdot)$ for each $n \geq 1$, we may conclude that the limit curve is a solution of \eqref{eq:Multiagent} driven by $\Ab_N(\cdot) \in \Gpazo_{\eta}(\tau,\mu)$, hence $\xb(\cdot) \in \Spazo_{(\tau,\mu)}^{\eta}(\Kb^0)$.  
\end{proof}

In the following lemma, we prove that the diameter of every solution of \eqref{eq:Multiagent} generated by a signal in $\Gpazo_{\eta}(\tau,\mu)$ must decrease strictly over every time interval of length $\tau >0$.  

\begin{lem}[Strict diameter decay under scrambling persistence]  
\label{lem:DiameterDecay}
Given a pair $(\tau,\mu) \in \R_+^* \times (0,1]$ and a 
compact set $\Kb^0 \subset (\R^d)^N$, it holds that 
%
\begin{equation*}
\Dpazo(\xb(t+\tau)) < \Dpazo(\xb(t))
\end{equation*}
for all times $t \geq 0$ and every curve $\xb(\cdot) \in \Spazo_{(\tau,\mu)}^{\eta}(\Kb^0)$ satisfying $\Dpazo(\xb(0)) > 0$.
\end{lem}

\begin{proof}
First, observe that the diameter function $t \in \R_+ \mapsto \Dpazo(\xb(t)) \in \R_+$ is nonincreasing along elements of $\Spazo_{(\tau,\mu)}^{\eta}(\Kb^0)$ as a simple consequence of Proposition~\ref{prop:NormEst}. In the sequel, given $t \geq 0$ at which $\Dpazo(\xb(t)) > 0$, we let
\begin{equation*}
\Pi(t) := \underset{(i,j)\in \{1,\dots,N\}^2}{\mathrm{argmax}} |x_i(t) - x_j(t)|
\end{equation*}
be the collection of all pairs that realise the diameter. We also note that the diameter is differentiable almost everywhere as the pointwise maximum of a collection of equi-Lipschitz maps, and it follows from Danskin's theorem (see e.g. \cite[Chapter III]{Danskin1967}) that
\begin{equation*}
\begin{aligned}
& \frac{1}{2} \derv{}{t} \Dpazo(\xb(t))^2 \\
& = \max_{(i,j) \in \Pi(t)} \langle \dot x_i(t) - \dot x_j(t) , x_i(t) - x_j(t) \rangle \\
& = \max_{(i,j) \in \Pi(t)} \frac{1}{N} \sum_{k=1}^N a_{ik}(s) \phi(|x_i(s) - x_k(s)|) \\
& \hspace{2.7cm} \times \langle x_k(s) - x_i(s) , x_i(s) - x_j(s) \rangle \\
& \hspace{0.7cm} - \frac{1}{N} \sum_{k=1}^N a_{jk}(s) \phi(|x_i(s) - x_k(s)|) \\
& \hspace{2.7cm} \times \langle x_k(s) - x_j(s) , x_i(s) - x_j(s) \rangle \\
& \geq -2C_{\phi} \Dpazo(\xb(t))^2
\end{aligned}
\end{equation*}
which implies up to an application Gr\"onwall's lemma that the diameter of any solution of \eqref{eq:Multiagent} cannot vanish in finite-time if it is initially non-zero. 

\paragraph*{Step 1 -- The case of switching maximising pairs} 

In order to show that the diameter is strictly contracting over the interval $[t,t+\tau]$, we first consider the situation
\begin{equation}
\label{eq:Intersection}
\bigcap_{s \in [t,t+\tau]} \Pi(s) = \emptyset, 
\end{equation}
namely we assume that it is not possible to find a pair of indices $(i,j) \in \{1,\dots,N\}^2$ that maximises the diameter over the whole interval $[t,t+\tau]$. In that case, there must exist some $\delta \in (0,\tau]$ for which $\Pi(t) \cap \Pi(t+\delta) = \emptyset$, whence
\begin{equation*}
\begin{aligned}
|x_i(t+\delta) - x_j(t+\delta)| & < \Dpazo(\xb(t+\delta)) \\
& \leq \Dpazo(\xb(t)) = |x_i(t) - x_j(t)|
\end{aligned}
\end{equation*}
for every pair $(i,j) \in \Pi(t)$. As the distances between all the agents realising the diameter at time $t \geq 0$ shrank strictly over the interval $[t,t+\delta]$, it either follows that the diameter has strictly decreased, or that there exists a pair of elements $(i_{\delta},j_{\delta}) \in \Pi(t+\delta)$ which achieves the value $\Dpazo(\xb(t))$. As the diameter of a collection of points coincides with that of its convex hull, this may only happen provided
\begin{equation}
\label{eq:DeltaPair0}
x_{i_{\delta}}(t+\delta) = x_i(t) \qquad \text{and} \qquad  x_{j_{\delta}}(t+\delta) = x_j(t)
\end{equation}
for some pair $(i,j) \in \Pi(t)$. One can show, however, that this situation cannot occur. To see why, observe that, by Proposition~\ref{prop:Diameter}, 
\begin{equation}
\label{eq:ConvexHullDerivative}
\begin{aligned}
& \derv{}{s} \langle x_i(t) - x_{i_{\delta}}(s) , x_i(t) - x_j(t) \rangle \\
& \hspace{0.5cm} = \frac{1}{N} \sum_{k=1}^N a_{i_{\delta}k}(s) \phi(|x_{i_{\delta}}(s) - x_k(s)|) \\
& \hspace{3.4cm}  \times \langle x_{i_{\delta}}(s) - x_k(s) , x_i(t) - x_j(t) \rangle \\
& \hspace{0.5cm} \geq \frac{1}{N} \sum_{l=1}^N a_{i_{\delta}k}(s) \phi(|x_{i_{\delta}}(s) - x_k(s)|) \\
& \hspace{3.4cm} \times \langle x_{i_{\delta}}(s) - x_i(t) , x_i(t) - x_j(t) \rangle \\
& \hspace{0.5cm} \geq - C_{\phi} \langle x_i(t) - x_{k_{ij}}(s) , x_i(t) - x_j(t) \rangle
\end{aligned}
\end{equation}
which, up to an application of Gr\"onwall's lemma, yields
\begin{equation}
\label{eq:DeltaPair1}
\begin{aligned}
\langle x_i(t) - x_{i_{\delta}} &(t+\delta) , x_i(t) - x_j(t) \rangle \\
& \geq \Big( \langle x_i(t) - x_{i_{\delta}}(t) , x_i(t) - x_j(t) \rangle \Big) e^{-C_{\phi}\delta}. 
\end{aligned}
\end{equation}
One may likewise show that 
\begin{equation}
\label{eq:DeltaPair2}
\begin{aligned}
\langle x_{j_{\delta}} (t+\delta) &- x_j(t) , x_i(t) - x_j(t) \rangle \\
& \geq \Big( \langle x_{j_{\delta}} (t) - x_j(t) , x_i(t) - x_j(t) \rangle \Big) e^{-C_{\phi}\delta}, 
\end{aligned}
\end{equation}
at which point it is possible to conclude by recalling that $(i_{\delta},j_{\delta}) \notin \Pi(t)$, which together with \eqref{eq:DeltaPair1} and \eqref{eq:DeltaPair2} combined with Proposition~\ref{prop:Diameter}, imply that either  
\begin{equation*}
\langle x_i(t) - x_{i_{\delta}} (t+\delta) , x_i(t) - x_j(t) \rangle > 0
\end{equation*}
or 
\begin{equation*}
\langle x_j(t) - x_{j_{\delta}} (t+\delta) , x_i(t) - x_j(t) \rangle < 0,
\end{equation*}
which is incompatible with \eqref{eq:DeltaPair0}. 

\paragraph*{Step 2 -- The case of a fixed maximising pair}

In this second step, we assume that
\begin{equation*}
\bigcap_{s \in [t,t+\tau]} \Pi(s) \neq \emptyset
\end{equation*}
so that there exists a pair of indices $(i,j) \in \Pi(t)$ such that
\begin{equation*}
\Dpazo(\xb(s)) = |x_i(s) - x_j(s)|
\end{equation*}
for all times $s \in [t,t+\tau]$. Suppose now by contradiction that $s \in [t,t+\tau] \mapsto \Dpazo(\xb(s)) \in \R_+^*$ is constant, which by what precedes would amount to having
\begin{equation}
\label{eq:DerivativeMaxPair}
\begin{aligned}
& \frac{1}{2} \derv{}{t} |x_i(t) - x_j(t)|^2 \\
& \hspace{0.7cm} = \langle \dot{x}_i(s) - \dot{x}_j(s) , x_i(s) - x_j(s) \rangle \\
& \hspace{0.7cm} = \frac{1}{N} \sum_{k=1}^N a_{ik}(s) \phi(|x_i(s) - x_k(s)|) \\
& \hspace{2.7cm} \times \langle x_k(s) - x_i(s) , x_i(s) - x_j(s) \rangle \\
& \hspace{0.7cm} - \frac{1}{N} \sum_{k=1}^N a_{jk}(s) \phi(|x_i(s) - x_k(s)|) \\
& \hspace{2.7cm} \times \langle x_k(s) - x_j(s) , x_i(s) - x_j(s) \rangle \\
& \hspace{0.7cm} = 0
\end{aligned}
\end{equation}
for almost every $s \in [t,t+\tau]$. Then, it follows from Hypotheses~\ref{hyp:H} and Proposition~\ref{prop:Diameter} that all the summands in the right-hand side of the previous expression are nonpositive, with the scalar products being equal to zero if and only if $x_k(s) = x_i(s)$ or $x_k(s) = x_j(s)$, respectively. Whence, this further implies that
\begin{equation*}
\left\{
\begin{aligned}
a_{ik}(s) & |x_k(s) - x_i(s)| = 0, \\
a_{jk}(s) & |x_k(s) - x_j(s)| = 0, 
\end{aligned}
\right.
\end{equation*}
for almost every $s \in [t,t+\tau]$ and each $k \in \{1,\dots,N\}$, at which point it necessarily holds
\begin{equation*}
x_i(\cdot) = x_i(t) \qquad \text{and} \qquad x_j(\cdot) = x_j(t)
\end{equation*}
over $[t,t+\tau]$ for that particular pair $(i,j) \in \cap_{s \in [t,t+\tau]}\Pi(s)$.

Presently, let us note that since $\xb(\cdot) \in \Spazo_{(\tau,\mu)}^{\eta}(\Kb^0)$, there must exist some index $k_{ij} \in \{1,\dots,N\}$ such that 
\begin{equation*}
\min\bigg\{ \frac{1}{\tau} \INTSeg{a_{ik_{ij}}(s)}{s}{t}{t+\tau} , \frac{1}{\tau} \INTSeg{a_{jk_{ij}}(s)}{s}{t}{t+\tau} \bigg\} \geq \mu. 
\end{equation*}
Yet, by reproducing the computations in \eqref{eq:DerivativeMaxPair} while resorting yet again to Proposition~\ref{prop:Diameter}, one may show that 
\begin{equation}
\label{eq:DerivativeMaxPairBis}
\left\{
\begin{aligned}
& \INTSeg{a_{ik_{ij}}(s) \langle x_i(t) - x_{k_{ij}}(s) , x_i(t) - x_j(t) \rangle}{s}{t}{t+\tau} = 0, \\
& \INTSeg{a_{jk_{ij}}(s) \langle x_{k_{ij}}(s) - x_j(t) , x_i(t) - x_j(t) \rangle}{s}{t}{t+\tau} = 0. 
\end{aligned}
\right.
\end{equation}
At this stage, either of the following two situations may occur. If $k_{ij} \in \{i,j \}$, say for instance $k_{ij} = j$, we may directly conclude as the resulting identity
\begin{equation*}
\begin{aligned}
& \INTSeg{a_{ik_{ij}}(s) \langle x_i(t) - x_{k_{ij}}(s) , x_i(t) - x_j(t) \rangle}{s}{t}{t+\tau} \\
& \hspace{4.4cm} \geq \mu \, |x_i(t) - x_j(t)|^2 > 0
\end{aligned}
\end{equation*}
violates \eqref{eq:DerivativeMaxPairBis}. If on the other hand $k_{ij} \in \{1,\dots,N\} \setminus \{i,j\}$, we can still produce a contradiction with \eqref{eq:DerivativeMaxPairBis} by repeating the computation in \eqref{eq:DeltaPair0} and then applying Gr\"onwall's lemma to obtain that
\begin{equation}
\label{eq:ScalarProduct1}
\begin{aligned}
\langle x_i(t) - x_{k_{ij}} &(s) , x_i(t) - x_j(t) \rangle \\
& \geq \Big( \langle x_i(t) - x_{k_{ij}}(t) , x_i(t) - x_j(t) \rangle \Big) e^{C_{\phi}(t-s)}
\end{aligned}
\end{equation}
for all $s \in [t,t+\tau]$, and similarly 
\begin{equation}
\label{eq:ScalarProduct2}
\begin{aligned}
\hspace{-0.1cm} \langle x_{k_{ij}}(s) - x_j&(t) , x_i(t) - x_j(t) \rangle \\
& \geq \Big( \langle x_{k_{ij}}(t) - x_j(t) , x_i(t) - x_j(t) \rangle \Big) e^{C_{\phi}(t-s)}. 
\end{aligned}
\end{equation}
Now, since $x_{k_{ij}}(t)$ cannot be simultaneously equal to $x_i(t)$ and $x_j(t)$, it subsequently stems from \eqref{eq:ScalarProduct1} and \eqref{eq:ScalarProduct2} that
\begin{equation*}
\begin{aligned}
d_{ij} := \max \bigg\{ & \min_{s \in [t,t+\tau]} \langle x_i(t) - x_{k_{ij}}(s) , x_i(t) - x_j(t) \rangle, \\
& \min_{s \in [t,t+\tau]} \langle x_{k_{ij}}(s) - x_j(t) , x_i(t) - x_j(t) \rangle \bigg\} > 0.
\end{aligned}
\end{equation*}
But this fact in is incompatible with \eqref{eq:DerivativeMaxPairBis}, as it imposes that at least one of the integrals therein is positive. 
\end{proof}

We are finally ready to prove our first main result, Theorem~\ref{thm:Main1}, which is based on the preliminary results derived above in Lemma~\ref{lem:CompactnessSolutionSet} and  Lemma~\ref{lem:DiameterDecay}.

\begin{proof}[Proof of Theorem~\ref{thm:Main1}]
As in our preparatory results, the argument goes by contradiction. Indeed, suppose at first that there exists some $\kappa \in (0,1)$ such that every curve $\xb(\cdot) \in \Spazo_{(\tau,\mu)}^{\eta}(\Kb^0)$ complies with the decay estimate  
\begin{equation}
\label{eq:DissipEta}
\Dpazo(\xb(\tau)) \leq \kappa  \Dpazo(\xb^0).
\end{equation}
We then claim that \eqref{eq:ThmDiamDissip} holds. Indeed, note first that one may express any time $t \geq 0$ as $t=n\tau+s$ for some $n \geq 1$ and $s\in [0,\tau)$, and observe that set $\Gpazo_{\eta}(\tau,\mu)$ of scrambling-persistent signals is invariant under positive time shifts by construction. Moreover, up to replacing the compact set $\Kb^0 \subset (\R^d)^N$ by its convex hull, we may assume that the latter is invariant under the dynamics of \eqref{eq:Multiagent} by Proposition~\ref{prop:NormEst}. Hence, it holds that 
\begin{equation*}
\begin{aligned}
\Dpazo(\xb(t)) & \leq \Dpazo(\xb(n \tau)) \\
& \leq \kappa^n \Dpazo(\xb^0) \\
& \leq \alpha \Dpazo(\xb^0) \exp(-\gamma t)
\end{aligned}
\end{equation*}
for every $\xb^0 \in \Kb^0$, with $\alpha := 1/\kappa^{\tau}$ and $\gamma :=-\log(\kappa)/\tau$. 

Assume now by contradiction that for every $\kappa \in (0,1)$, there exists a curve $\xb_{\kappa}(\cdot) \in \Spazo_{(\tau,\mu)}^{\eta}(\Kb^0)$ which violates \eqref{eq:DissipEta}. In particular, there exists then a sequences of trajectories $(\xb_n(\cdot)) \subset \Spazo_{(\tau,\mu)}^{\eta}(\Kb^0)$ whose elements satisfy
\begin{equation}
\label{eq:DiameterContradictionIneq}
\Dpazo(\xb_n(\tau)) > \frac{n-1}n  \Dpazo(\xb_n(0))>0
\end{equation}
for each $n \geq 1$, and it follows from the compactness result established in Lemma~\ref{lem:CompactnessSolutionSet} that 
\begin{equation*}
\sup_{t \in [0,\tau]} |\xb(t) - \xb_n(t)| ~\underset{n \to \infty}{\longrightarrow}~ 0
\end{equation*}
for some limit curve $\xb(\cdot) \in \Spazo_{(\tau,\mu)}^{\eta}(\Kb^0)$, along an adequate subsequence. Then, by letting $n \to +\infty$ in \eqref{eq:DiameterContradictionIneq}, we would further obtain that 
\begin{equation*}
\Dpazo(\xb(\tau)) \geq \Dpazo(\xb(0)), 
\end{equation*}
which violates the diameter decay established for elements of $\Spazo_{(\tau,\mu)}^{\eta}(\Kb^0)$ in Lemma~\ref{lem:DiameterDecay}, provided that $\Dpazo(\xb(0)) > 0$.

If on the other hand, $\Dpazo(\xb(0)) = 0$, we may reason as follows. Notice first that $\lim_{n\to \infty}\bar \xb_n(0)=\bar \xb(0)$, and consider the curves $\yb_n(\cdot) \in \Lip(\R_+,(\R^d)^N)$ defined by
\begin{equation}
\label{eq:ybDef}
\yb_n(t) := \frac{\xb_n(t)-\bar{\xb}_n(0)}{\Dpazo(\xb_n(0))},  
\end{equation}
for all times $t \geq 1$ and each $n \geq 1$, which correspond the dilations of $\xb_n(\cdot)$ around the point $\bar{\xb}_n(0)$ by a factor $1/ \Dpazo(\xb_n(0))$. It may then be checked that the components of $\yb_n(\cdot)$ solve the system of equation
\begin{align*}
    &\dot y_i^n(t) = \frac{1}{N} \sum_{j=1}^N a^n_{ij}(t) \phi \big( \Dpazo(\xb_n(0))|y_i^n(t)-y_j^n(t)| \big) \\
    & \hspace{4.95cm} \times (y_j^n(t)-y_i^n(t)),
\end{align*}
wherein $\Ab_N^n(\cdot) := (a^n_{ij}(\cdot))_{1 \leq i,j \leq N}\in \Gpazo_{\eta}(\tau,\mu)$ are the signals generating the curves $\xb_n(\cdot) \in \Spazo_{(\tau,\mu)}^{\eta}(\Kb^0)$. By repeating the argument of Lemma~\ref{lem:CompactnessSolutionSet}, taking 
\begin{equation*}
\Qb^0 := \Big\{ \yb \in (\R^d)^N ~\, \textnormal{s.t.}~ \max_{1 \leq i \leq N} |y_i| \leq 1 \Big\},    
\end{equation*}
as compact set of initial conditions, one can show that, up to an extraction, the sequence $(\yb_n(\cdot)) \subset \Spazo_{(\tau,\mu)}^{\eta}(\Qb^0)$ converges uniformly on compact subsets of $\R_+$ to some limit curve $\yb(\cdot)$ solution of the dynamics
\begin{align*}
\dot y_{i}(t) = \frac{1}{N} \sum_{j=1}^N a_{ij}(t) \phi(0)(y_{j}(t)-y_{i}(t)),
\end{align*}
for some $\Ab_N(\cdot) \in \Gpazo_{\eta}(\tau,\mu)$. Moreover, 
it follows from \eqref{eq:DiameterContradictionIneq} that $\Dpazo(\yb_n(0))=1$ for every $n\ge1$, yielding that  $\Dpazo(\yb(0))=1$. 
Thanks to \eqref{eq:ybDef}, we also have that
\begin{equation*}
\Dpazo(\yb(\tau))\ge \liminf_{n\to \infty}\frac{\Dpazo(\xb_n(\tau))}{\Dpazo(\xb_n(0))}=1,
\end{equation*}
which is in contradiction with Lemma~\ref{lem:DiameterDecay} applied to $\yb(\cdot) \in \Lip(\R_+,(\R^d)^N)$, whose dynamics is of the form \eqref{eq:Multiagent} wherein the nonlinear kernel $\phi \in \Lip_{\loc}(\R_+,\R_+^*)$ has been replaced by the constant $\phi(0) \in \R_+^*$. 
\end{proof}


\section{Exponential variance decay for connectivity-persistent topologies}
\label{section:Variance}

In this section, we turn our attention to the following class of linear balanced multiagent dynamics
\begin{equation}
\label{eq:MultiagentLin}
\left\{
\begin{aligned}
& \dot x_i(t) = \frac{1}{N} \sum_{j=1}^N a_{ij}(t)(x_j(t)-x_i(t)), \\
& x_i(0) = x_i^0.
\end{aligned}
\right.
\end{equation}
By leveraging the compactification technique developed hereinabove, we establish a simpler, yet useful exponential contraction result for the variance of such systems. Let us mention that a compactification argument with a similar point of view was used in \cite[Theorem~7 and Lemma~10]{Chaillet2008} to study the stabilisation of persistently excited linear systems in the neutrally stable case.

\begin{thm}[Exponential variance decay for connectivity-persistent balanced interaction topologies]
\label{thm:Main2}
Given a pair $(\tau,\mu) \in \R_+^* \times (0,1]$ and a compact set $\Kb^0 \subset \R^N$, there exist positive constants $\alpha,\gamma > 0$ such that
\begin{equation}
\label{eq:ThmVarDissip}
\Vpazo(\xb(t)) \leq \alpha \Vpazo(\xb^0) \exp(-\gamma t)
\end{equation}
for all $t \geq 0$ and any solution $\xb(\cdot) \in \Lip(\R_+,(\R^d)^N)$ of \eqref{eq:Multiagent} driven by a signal $\Ab_N(\cdot) \in \Gpazo_{\lambda_2}(\tau,\mu)$. 
\end{thm}

The proof of this theorem is essentially the same as that of Theorem~\ref{thm:Main1}, and hinges on the following compactness and strict decay property enjoyed by solutions of \eqref{eq:MultiagentLin}.

\begin{lem}[Compactness of solution sets]
\label{lem:CompactnessSolutionSetBis}
For every $(\tau,\mu) \in \R_+^* \times (0,1]$ and each compact set $\Kb^0 \subset (\R^d)^N$, the collection of curves 
\begin{equation*}
\begin{aligned}
\Spazo_{(\tau,\mu)}^{\lambda_2}(\Kb^0) := \bigg\{ & \xb(\cdot) \in \Lip(\R_+,(\R^d)^N) ~ \text{solving \eqref{eq:MultiagentLin}} \\
& \; \text{with $\xb^0 \in \Kb^0$ and $\Ab_N(\cdot) \in \Gpazo_{\lambda_2}(\tau,\mu)$}\bigg\}
\end{aligned}
\end{equation*}
is a compact subset of $C^0(\R_+,(\R^d)^N)$ for the topology of uniform convergence on compact sets. 
\end{lem}

\begin{proof}
The proof essentially follows the same arguments as that of Lemma~\ref{lem:CompactnessSolutionSet}. 
\end{proof}

\begin{lem}[Strict variance decay under connectivity persistence]
\label{lem:VarianceDecay}
Given a pair $(\tau,\mu) \in \R_+^* \times (0,1]$ and a compact set $\Kb^0 \subset (\R^d)^N$, it holds that 
\begin{equation*}
\Vpazo(\xb(t+\tau)) < \Vpazo(\xb(t))
\end{equation*}
for all times $t \geq 0$ and every curve $\xb(\cdot) \in \Spazo_{(\tau,\mu)}^{\lambda_2}(\Kb^0)$ satisfying $\Vpazo(\xb(0)) > 0$.
\end{lem}

\begin{proof}
First, we make the observation that the average of any curve $\xb(\cdot) \in \Spazo_{(\tau,\mu)}^{\lambda_2}(\Kb^0)$ is constant, since 
\begin{equation*}
\begin{aligned}
\derv{}{t} \bar{\xb}(t) & = \frac{1}{N^2} \sum_{i,j=1}^N a_{ij}(t) (x_j(t) - x_i(t)) \\
& = \frac{1}{N} \sum_{i=1}^N \bigg( \frac{1}{N} \sum_{j=1}^N (a_{ji}(t) - a_{ij}(t)) \bigg) x_i(t) = 0, 
\end{aligned}
\end{equation*}
and also that the variance is monotonically nonincreasing along solutions of \eqref{eq:MultiagentLin}. Indeed, one may easily check that
\begin{equation*}
\begin{aligned}
\frac{1}{2} \derv{}{t} \Vpazo(\xb(t)) & = \derv{}{t} \frac{1}{N} \sum_{i=1}^N |x_i(t) - \bar{\xb}(t)|^2 \\
& = \frac{1}{N^2} \sum_{i,j=1}^N a_{ij}(t) \langle x_i(t) - \bar{\xb}^0 , x_j(t) - x_i(t) \rangle \\
& = - \frac{1}{2N^2} \sum_{i,j=1}^N a_{ij}(t) |x_i(t) - x_j(t)|^2 \leq 0
\end{aligned}
\end{equation*}
where we again used the fact that the matrices $\Ab_N(t) \in \Lpazo(\R^N)$ are balanced for almost every $t \geq 0$. By contradiction, suppose that the variance is constant over the time interval $[t,t+\tau]$, so that
\begin{equation*}
\derv{}{s} \Vpazo(\xb(s)) = - \frac{1}{2N^2} \sum_{i,j=1}^N a_{ij}(s) |x_i(s) - x_j(s)|^2 = 0
\end{equation*}
for almost every $s \in [t,t+\tau]$. As all the summands appearing in the right-hand side of the previous expression are nonnegative, this would entail that 
\begin{equation*}
a_{ij}(s) |x_i(s) - x_j(s)|^2 = 0
\end{equation*}
for all times $s \in [t,t+\tau]$ and each $i \in \{1,\dots,N\}$, and consequently that $x_i(\cdot) = x_i(t)$ over the whole interval. However, it would then follow from the lower estimate \eqref{eq:Lambda2Charac} characterising the algebraic connectivity that
\begin{equation*}
\begin{aligned}
& \Vpazo(\xb(t+\tau)) \\
& = \Vpazo(\xb(t)) - \frac{1}{2N^2} \sum_{i,j=1}^N \INTSeg{a_{ij}(s)|x_i(s) - x_j(s)|^2}{s}{t}{t+\tau} \\
& = \Vpazo(\xb(t)) - \frac{1}{2N^2} \sum_{i,j=1}^N \bigg( \INTSeg{a_{ij}(s)}{s}{t}{t+\tau} \bigg) |x_i(t) - x_j(t)|^2 \\
& \leq (1-\tau \mu) \Vpazo(\xb(t)) < \Vpazo(\xb(t)) 
\end{aligned}
\end{equation*}
which leads to a contradiction.  
\end{proof}


\bibliographystyle{ieeetr}
{\footnotesize
\bibliography{ControlWassersteinBib}
}
\end{document}